\documentclass{amsart}
\usepackage{amsmath}
\usepackage[usenames,dvipsnames]{color}
\numberwithin{equation}{section}
\usepackage{amssymb}
\usepackage{comment}
\usepackage[all]{xy}
\usepackage{mathtools}
\definecolor{Vino}{rgb}{0.356,0,0}
\definecolor{Ruta}{rgb}{0.309, 0.459, 0.208}
\usepackage[colorlinks,linkcolor=Vino,citecolor=Ruta]{hyperref}
\let\cal\mathcal
\def\Ascr{{\cal A}}
\def\Bscr{{\cal B}}
\def\Cscr{{\cal C}}
\def\Dscr{{\cal D}}
\def\Escr{{\cal E}}
\def\Fscr{{\cal F}}
\def\Gscr{{\cal G}}
\def\Hscr{{\cal H}}

\def\Kscr{{\cal K}}
\def\Lscr{{\cal L}}
\def\Mscr{{\cal M}}

\def\Oscr{{\cal O}}
\def\Pscr{{\cal P}}

\def\Sscr{{\cal S}}
\def\Tscr{{\cal T}}

\def\Wscr{{\cal W}}
\def\Xscr{{\cal X}}

\let\blb\mathbb
\def\CC{{\blb C}}

\def \PP{{\blb P}}

\def \ZZ{{\blb Z}}

\def \RR{{\blb R}}

\def\SS{{\blb S}}

\def\id{\text{id}}
\def\Id{\operatorname{id}}

\def\quot{/\!\!/}

\def\coh{\mathop{\text{\upshape{coh}}}}

\def\Spec{\operatorname {Spec}}

\def\Hom{\operatorname {Hom}}

\def\End{\operatorname {End}}
\def\RHom{\operatorname {RHom}}
\def\uRHom{\operatorname {R\mathcal{H}\mathit{om}}}

\def\End{\operatorname {End}}

\def\id{{\operatorname {id}}}

\def\r{\rightarrow}

\DeclareMathOperator{\Ind}{Ind}

\def\la{\langle}
\def\ra{\rangle}

\DeclareMathOperator\RInd{RInd}
\def\deltaC{\xi}

\newtheorem{lemma}{Lemma}[section]
\newtheorem{proposition}[lemma]{Proposition}
\newtheorem{theorem}[lemma]{Theorem}
\newtheorem{corollary}[lemma]{Corollary}

\newtheorem{observation}[lemma]{Observation}

\theoremstyle{definition}

\newtheorem{example}[lemma]{Example}
\newtheorem{definition}[lemma]{Definition}

{

}

\theoremstyle{remark}

\newtheorem{remark}[lemma]{Remark}

\newdimen\uboxsep \uboxsep=1ex
\def\uboxn#1{\vtop to 0pt{\hrule height 0pt depth 0pt\vskip\uboxsep
\hbox to 0pt{\hss #1\hss}\vss}}

\def\uboxs#1{\vbox to 0pt{\vss\hbox to 0pt{\hss #1\hss}
\vskip\uboxsep\hrule height 0pt depth 0pt}}

\let\oldmarginpar\marginpar
\long\def\marginpar#1{\oldmarginpar{\raggedright \tiny \sf\baselineskip 0pt \lineskip 0pt #1}}

\def\vect{\operatorname{vec}}
\def\Sym{\operatorname{Sym}}

\def\cone{\operatorname{cone}}
\def\Perf{\operatorname{Perf}}
\def\Ob{\operatorname{Ob}}
\edef\savedlineskip{\lineskip \the\lineskip}
\title{A class of perverse schobers in Geometric Invariant Theory} 
\author[\v{S}pela \v{S}penko and Michel Van den Bergh]{\v{S}pela \v{S}penko and Michel
  Van den Bergh} 
\thanks{The first author is a FWO $[$PEGASUS$]^2$
  Marie Sk\l odowska-Curie fellow at the Free University of Brussels
  (funded by the European Union Horizon 2020 research and innovation
  program under the Marie Sk\l odowska-Curie grant agreement No
  665501 with the Research Foundation Flanders (FWO))}  
\address{Departement Wiskunde, Vrije Universiteit
  Brussel, Pleinlaan $2$, B-1050 Elsene}
\email[]{spela.spenko@vub.be}
\address{Departement WNI, Universiteit Hasselt, Universitaire Campus \\
  B-3590 Diepenbeek} 
\email[]{michel.vandenbergh@uhasselt.be}
\thanks{The second author is a senior researcher at the Research
  Foundation Flanders (FWO).  While working on this project he was
  supported by the FWO grant G0D8616N: ``Hochschild cohomology and
  deformation theory of triangulated categories''.}
\begin{document}
\begin{abstract}
  Perverse schobers are categorifications of perverse sheaves.  We
  construct a perverse schober on a partial compactification of the
  stringy K\"ahler moduli space (SKMS) associated by Halpern-Leistner
  and Sam to a quasi-symmetric representation $X$ of a reductive
  group $G$, extending the local system of triangulated categories exhibited by them. 
The triangulated categories appearing in our
  perverse schober are subcategories of the derived category of
  the quotient stack $X/G$.
\end{abstract}
\maketitle
\section{Introduction}   Perverse sheaves appear as (derived)
solution spaces to systems of linear differential equations with nice (``regular'')
singularities and perverse schobers are categorifications of perverse
sheaves. The simplest perverse sheaves on a connected complex manifold~$M$ are
the local systems (suitably shifted) and these correspond to representations of
$\pi_1(M,x)$ where $x$ is a base point. The corresponding categorified
notion is a triangulated category $\Ascr$ with an $\pi_1(M,x)$-action
by autoequivalences.  We call this \emph{a local system of
  triangulated categories} on $M$ and we refer to $\Ascr$ as the
\emph{fiber} in $x$ of the local system.

\medskip

In general a perverse sheaf on $M$ is only a local system on some dense open $U\subset M$. Likewise a perverse schober is a suitable
``extension to $M$''  of a local system of triangulated categories on some $U$.
In an ideal world a  perverse schober on $M$ would just be a perverse sheaf of triangulated categories on $M$. Unfortunately
it is not known how to make
such an approach work in general, the main culprit being the absence of a sensible notion of complexes of triangulated categories.

\medskip

On the other hand, there are many cases where the category of perverse sheaves admits a combinatorial description
and one may try to categorify the latter. This was the approach taken by Kapranov and Schechtman
in the foundational paper~\cite{KapranovSchechtmanSchobers} where amongst others they made the
beautiful observation that perverse sheaves on a punctured disk are naturally categorified by
spherical functors. 

\medskip

A much more general instance where a combinatorial description of the category of perverse sheaves is known  \cite{KapranovSchechtman} 
is that of a complex affine space stratified by a complexified real hyperplane arrangement  $\Hscr$
and this leads to a corresponding definition of perverse schobers \cite{BondalKapranovSchechtman}. See \S\ref{sec:schobers} for more details.
Since there is some flexibility in  what one wants from a perverse schober (see \S\ref{sec:spherical}) we use this notion somewhat loosely and we refer to the specific notion introduced in \cite{BondalKapranovSchechtman}
as an \emph{$\Hscr$-schober} (this is the terminology in loc.\ cit.).

\medskip

Hyperplane arrangements are common in representation theory and geometric invariant theory and they are often accompanied by wall-crossing formulas (``reflections''). If these
can be categorified in some way to a local system of categories on the open stratum then one may hope to enhance the categorification further to a perverse schober on the whole space.
See \cite{BondalKapranovSchechtman} for some far reaching conjectures in this direction related to \cite{Bez}.
In the context of geometric invariant theory perverse schobers on a punctured disc were associated to a ``balanced wall-crossing'' in VGIT by Donovan \cite{Donovan2,Donovan}, building on earlier
work by Halpern-Leistner and Schipman \cite{halpern2016autoequivalences}. 

\medskip

In fact the main example considered in \cite{Donovan} is a very
special case of hyperplane arrangements constructed by
Halpern-Leistner and Sam in \cite{HLSam} (as noted e.g.\ in
\cite[Remark before \S1.2]{Donovan}). The main purpose of this note
will be to construct perverse schobers on the latter in complete
generality.

\medskip

The input in the work in \cite{HLSam} is a connected reductive group $G$ with
maximal torus $T$ and Weyl group $\Wscr$ and a 
$G$-representation $W$ which is \emph{quasi-symmetric}. By the latter
we mean that if $(\beta_i)_i\in X(T)$ are the weights of $W$ then
for every line $\ell\subset X(T)_\RR$ through the origin we have
$\sum_{\beta_i\in\ell} \beta_i=0$. This is in particular the case if $W\cong W^\ast$,
i.e.\ if $W$ is equipped with a non-degenerate $G$-invariant bilinear form.
Below we also assume that the generic $T$-stabilizer is finite (i.e.\ $(\beta_i)_i$ spans $X(T)_\RR$). 

\medskip

Quasi-symmetric representations were introduced in \cite{SVdB}. It was
shown in loc.\ cit.\ that for such representations the GIT quotient
$X\quot G$ often admits a \emph{non-commutative crepant resolution}
\cite{VdB04} where $X=\Sym(W)=W^\ast$.  If $G$ is not semi-simple and hence $X(G)=X(T)^{\Wscr}\neq 0$ 
then we may also consider non-trivial linearizations of the $G$-action associated to
$0\neq \chi\in X(G)$. 
If $G$ is a torus and $\chi$ is chosen generically  then the
semi-stable locus $X^{ss,\chi}$ of~$X$ for such a linearization
yields a smooth Deligne-Mumford stack $X^{ss,\chi}/G$ which resolves $X\quot G$. 
Moreover this is often true for more general groups as well (see \cite[Proposition 2.1]{HLSam} for the precise conditions)
so let us consider this case. In~\cite{HLSam} Halpern-Leistner and Sam construct an 
$X(T)^\Wscr$-invariant hyperplane arrangement $\Hscr$ in $X(T)^\Wscr_\RR$ 
together with a  local system on $(X(T)^{\Wscr}_{\CC}-\Hscr_\CC)/X(T)^\Wscr$ with fiber
$D(X^{ss,\chi}/G)$. Mirror symmetry considerations let one think of
the quotient $(X(T)^\Wscr_{\CC}-\Hscr_{\CC})/X(T)^\Wscr$ as the so-called
\emph{stringy K\"ahler moduli space} (SKMS) of $X^{ss,\chi}/G$ when $X$ is a quasi-symmetric representation (in \cite{Kite} this is somewhat rigorously established when $G$ is a torus).

The actual construction of $\Hscr$ is a bit involved and we refer the reader to \S\ref{sec:SKMSqs}. Let us restrict
ourselves to giving the following simple example.
\begin{example}\cite{Donovan}
\label{ex:donovan}
Let $G=G_m$ be a 1-dimensional torus acting on a representation with weights $(a_i)_i\in \ZZ$ such that $\sum_i a_i= 0$
and such that not all $a_i$ are zero. Then the SKMS is 
$(X(G_m)_{\CC}-X(G_m))/X(G_m)=(\CC-\ZZ)/\ZZ {\cong} \PP^1_{\CC}-\{0,1,\infty\}$ (via $z\mapsto e^{2\pi i z}$).
In the case the weights are $-1,\ldots,-1,1,\ldots,1$ this example is considered in \cite{Donovan}. The resulting local
system of triangulated categories is related to the simplest case of a flop (see also Example \ref{ex:flober}).
\end{example}

In this paper we prove the following. 
\begin{theorem} \label{th:mainth}
The local system on the SKMS given in \cite{HLSam} extends to 
a perverse schober on the partial compactification of the SKMS given by $X(T)^{\Wscr}_{\CC}/X(T)^\Wscr$.
\end{theorem}

See Proposition \ref{prop:schober} for a much more precise statement. In the context of Example~\ref{ex:donovan}, Theorem \ref{th:mainth} means that we construct a perverse schober
of $\PP^1_{\CC}-\{0,\infty\}$ extending the local system on $\PP^1_{\CC}-\{0,1,\infty\}$. This is the same as what is achieved in \cite{Donovan} when the weights are $-1,\ldots,-1,1,\ldots,1$.

\medskip

Of course the reader will note that we have been rather sloppy in the formulation of Theorem \ref{th:mainth} since the SKMS is not actually the complement of a hyperplane arrangement in an affine 
space but rather a quotient of such. So what we actually do below is  construct an $\Hscr$-schober  on $X(T)^\Wscr_{\CC}$ which is $X(T)^\Wscr$-equivariant in an appropriate sense.
See Remark \ref{rem:stack}.
\section{Notation and conventions}
\label{sec:notation}
Let $k$ be an algebraically closed field of characteristic zero.
Throughout everything is linear over $k$.

Let $G$ be a {connected} 
reductive group. 
 Let $T\subset B\subset G$ be respectively a maximal torus and a Borel subgroup. We assume that the roots of $B$ are the negative roots.  
Half the sum of positive roots of $G$ is denoted by $\overline{\rho}$. 
Let $\Wscr=N(T)/T$ be the Weyl group of $T$. We denote the dominant cones in $X(T)$ and $Y(T)$
by $X(T)^+$ and $Y(T)^{+}$, respectively. The corresponding notations $X(T)^-$, $Y(T)^-$ denote the anti-dominant cones.
If $\chi$ in $X(T)_\RR$ and there exists $w\in \Wscr$ such that $w{\ast}\chi:=w(\chi+\overline{\rho})-\overline{\rho}\in X(T)^+_\RR$ then we write $\chi^+=w{\ast} \chi$.
Otherwise $\chi^+$ is undefined.
If $\chi\in X(T)^+$ then we denote $V(\chi):=\Ind_B^G\chi$. Note that $V(\chi)$ is the irreducible representation of $G$ with highest weight $\chi$. 
By $w_0$ we denote the longest element in $\Wscr$. In particular
$V(\chi)^\ast=V(-w_0\chi)$. If $\alpha$ is a simple root then the corresponding reflection is written as $s_\alpha$.
On $Y(T)_\RR$ we sometimes choose a positive definite $\Wscr$-invariant quadratic form and denote the corresponding norm by $\|\;\|$. 

Below $W$ will be a finite dimensional $G$-representation and we denote the $T$-weights of $W$ by $(\beta_i)_{i=1}^d\in X(T)$. We write $X=\Spec \Sym(W)\cong W^\ast$. We say that $W$ is \emph{unimodular} if  $\wedge^dW\cong k$. 
Most of our results will be for the case that $W$ is \emph{quasi-symmetric} (see \cite[\S1.6]{SVdB}), i.e. for all lines $\ell$ such that $0\in\ell \subset X(T)_\RR$ we have $\sum_{\beta_i\in \ell}\beta_i=0$. 
We impose this condition from \S\ref{sec:SKMSqs} on. 
For $\lambda\in Y(T)_\RR$ let $X^{\lambda,+}:=\Spec \Sym(W/K_\lambda)$  where $K_\lambda$ is the linear subspace of $W$ spanned by the weight vectors $w_j$ with $\la\lambda,\beta_j\ra >0$.

\medskip

Now we list some more general conventions. All modules are left modules. All stacks are algebraic (and in fact quotient stacks).
If $\Lambda$ is a ring then $D(\Lambda)$ is the unbounded derived category. If $\Xscr$ is a stack we denote by $D(\Xscr)$ the unbounded derived category of complexes
of $\Oscr_{\Xscr}$-modules with quasi-coherent cohomology. The notation $\Lambda^\circ$ denotes the opposite ring of $\Lambda$,  if $\Ascr$ is a category then $\Ascr^\circ$ is the opposite category.

If $j:Y\r X$ is a closed embedding then we confuse $\Oscr_Y$ with $j_\ast\Oscr_Y$. If in addition~$X$ is an affine $G$-variety and $Y$ is a $B$-variety then
we (severely) abuse notation by writing  $\RInd^G_B$
for the derived pushforward for $Y/B\xrightarrow{\cong} G\times^B Y/G\r X/G$.

In any kind of cell complex, cells (or faces,  cones,\dots ) are assumed to be relatively open. We refer to their closures as ``closed cells''. The ordering on cells is by $C\le C'$
iff $C\subset \overline{C'}$. \emph{Facets} are cells of codimension one. In the case of a polyhedron we regard the full polyhedron as a face.

If $\Ascr$ is a triangulated category closed under coproduct then we write $\Ascr^c$
for the category of compact objects in $\Ascr$. If
$\Sscr\subset \Ob(\Ascr)$ then the full subcategory of $\Ascr$
\emph{generated} by $\Sscr$ (notation $\langle \Sscr\rangle$) is the smallest
strict, full triangulated subcategory of~$\Ascr$ closed under
coproduct which contains $\Sscr$. We often implicitly use the fact that if $\Sscr$ consists of compact objects then $\langle \Sscr\rangle^c$ is \emph{classically generated} by $\Sscr$ \cite{Neeman3}; i.e.\ $\langle\Sscr\rangle^c$ is the smallest (strict) thick subcategory of $\Ascr$ which contains $\Sscr$. In particular
$\langle \Sscr\rangle^c=\langle \Sscr\rangle \cap \Ascr^c$. In general one must be quite careful with compact objects on stacks
\cite{HallNeemanRydh,HallRydh} but in the benign situation we consider there are no surprises. If $G$ is a reductive group
acting on an affine variety $X$ then $D(X/G)$ is compactly generated and the compact objects are the $G$-equivariant
perfect complexes (see e.g.\ \cite[Theorem 3.5.1]{SVdB3}).

If $\Bscr$ is a full subcategory of
$\Ascr$ then $\Bscr^\perp$ is the full subcategory of~$\Ascr$ spanned by the objects
$\{A\in\Ob(\Ascr)\mid \Hom(B,A)=0 \text{ for all } B\in
\Bscr\}$.

\section{Perverse schobers on affine hyperplane arrangements}
\label{sec:schobers}
\subsection{Introduction}
In this section we define perverse schobers on linear spaces stratified by affine hyperplane
arrangements.  This is achieved by 
 categorifying the Kapranov and Schechtman \cite{KapranovSchechtman} description of
perverse sheaves on hyperplane arrangements.  In the linear
case such a categorification was carried out in \cite{BondalKapranovSchechtman} and the affine case 
is just a straightforward generalization.
Since there is some ambiguity what the precise conditions on a perverse schober should be in this situation (see e.g.\ \S\ref{sec:spherical} below) 
the perverse schobers from \cite{BondalKapranovSchechtman} are more specifically referred to as \emph{$\Hscr$-schobers}.
\subsection{Definitions}
We first recall the description of Kapranov and Schechtman of perverse sheaves.
Let $\Hscr$ be an affine hyperplane arrangement in a finite dimensional real vector space~$V$.
The closures of the connected components of $\RR^n\setminus \Hscr$ are convex polytopes. Let $(\Cscr,\leq)$ denote the set of these polytopes partially ordered by 
  $C_1\le C_2$ iff $C_1\subset \overline{C_2}$. If $C_1$ and $C_2$ share a common face then there is a maximal one and we denote it by $C_1\wedge C_2$.
A triple of faces $(C_1,C_2,C_3)$ is {\em collinear} if there exists $C'\leq C_1,C_2,C_3$ and there exist points $c_i\in C_i$ such that $c_2\in[c_1,c_3]$.

\begin{theorem}\cite[Theorem 9.10]{KapranovSchechtman}
  The category 
   of perverse sheaves  on $V_\CC$ with respect to the stratification induced by $\Hscr_\CC$ is equivalent to the category
  of diagrams consisting of finite dimensional vector spaces $E_C$,
  $C\in\Cscr$, and linear maps $\gamma_{C'C}:E_{C'}\to E_C$,
  $\delta_{CC'}:E_C\to E_{C'}$ for $C'\leq C$ such that
  $(E_C,(\gamma_{C'C})_{C,C'})$ is a representation of $(\Cscr,\leq)$
  in $\vect(k)$ and $(E_C,(\delta_{CC'})_{CC'})$ a
  representation of $(\Cscr,\geq)$ in $\vect(k)$ and the following
  conditions are satisfied:
\begin{enumerate}
\item[(m)] $\gamma_{C'C}\delta_{CC'}=\id_{E_C}$ for $C'\leq C$. 
In particular, $\phi_{C_1C_2}:=\gamma_{C'C_2}\delta_{C_1C'}$ for $C'\leq C_1,C_2$ is well defined. 
\item[(i)] $\phi_{C_1C_2}$ is an isomorphism for every $C_1\neq C_2$ which
  are of the same dimension $d$ lying, lie in the same $d$-dimensional
  affine space and share a facet. 
\item[(t)] $\phi_{C_1C_3}=\phi_{C_2C_3}\phi_{C_1C_2}$ for collinear
  triples of faces $(C_1,C_2,C_3)$.
\end{enumerate}
\end{theorem}
Replacing vector spaces with categories and maps with functors we obtain the concept of an $\Hscr$-schober. Let us first specify  the meaning of the representation of $(\Cscr,\leq)$ in the (2-)category.

\begin{definition}
\label{def:representation}
  Let $\Tscr$ be a category (resp. 2-category). We say that
  $((T_C\in\Tscr)_{C\in \Cscr},(f_{C'C}:T_{C'}\to T_C)_{C'\leq C\in
    \Cscr})$
  is a representation of $(\Cscr,\leq)$ in the category
  (resp. $2$-category) $\Tscr$ if for $C''\leq C'\leq C$ one has
  $f_{C'C}f_{C''C'}=f_{C''C}$ (resp. there are isomorphisms $\kappa_{C''C'C}:f_{C'C}f_{C''C'}\to f_{C''C}$  satisfying the standard compatibility condition for
  $C'''\leq C''\leq C'\leq C$).
\end{definition}
\begin{remark}
  In \cite[Definition 3.3]{BondalKapranovSchechtman} the natural
  isomorphisms of functors as above 
  $\kappa_{C''C'C}:f_{C'C}f_{C''C'}\to f_{C''C}$ 
for
  $C''\leq C'\leq C$ are kept as part of the data, and this data is referred to as a ``triangulated $2$-functor''.
An alternative terminology for this concept is ``pseudo-functor''.
\end{remark}

\begin{definition}\label{def:schober}
  An $\Hscr$-schober on $V_\CC$  is given by triangulated categories $\Escr_C$, 
  $C\in \Cscr$,  adjoint
pairs of exact functors $(\delta_{CC'}:\Escr_C\to \Escr_{C'},\gamma_{C'C}:\Escr_{C'}\to \Escr_C)$ 
   for $C'\leq C$ such that
$(\Escr_C,(\delta_{C'C})_{C'C})$, 
is a representation of
$(\Cscr,\geq)$
in the $2$-category of
  triangulated categories satisfying  the following conditions:
 \begin{enumerate}
 \item[(M)] The unit of the adjunction $(\delta_{CC'},\gamma_{C'C})$
defines a natural isomorphism
   $\id_{\Escr_C}\xrightarrow{\cong}\gamma_{C'C}\delta_{CC'}$ for $C'\leq C$ , and
   thus $\phi_{C_1C_2}:=\gamma_{C'C_2}\delta_{C_1C'}$ for
   $C'\leq C_1,C_2$ is well defined up to canonical
natural   isomorphism.
\item[(I)]
$\phi_{C_1C_2}$ is an equivalence for every $C_1\neq C_2$ of the same dimension $d$ lying in the same $d$-dimensional affine space which share a facet.
\item[(T)]
For collinear triples of faces $(C_1,C_2,C_3)$ with common face $C_0$ the counit of the adjunction $(\delta_{C_0C_2},\gamma_{C_2C_0})$
 defines a natural isomorphism
$\phi_{C_2C_3}\phi_{C_1C_2}\xrightarrow{\cong} \phi_{C_1C_3}$.
\end{enumerate}
\end{definition}
\subsection{Equivariant \boldmath $\Hscr$-schobers}
Assume $V$ is equipped with an affine, $\Hscr$-preserving, group action by a group $\Gscr$.
In that case $\Cscr$ is of course also preserved. A \emph{$\Gscr$-action}
on an $\Hscr$-schober on $V_\CC$ is a collection of exact functors
\[
\phi_{g,C}:\Escr_C\r \Escr_{gC}
\]
for $g\in \Gscr$, $C\in \Cscr$, enhanced with natural isomorphisms $\phi_{h,gC}\phi_{g,C}\cong \phi_{hg,C}$ satisfying the obvious compatibility
for triple products in $\Gscr$. Moreover we should have pseudo-commutative diagrams for every $C'<C$:
\begin{equation}
\label{eq:deltadiagram}
\xymatrix{
\Escr_C\ar[r]^{\phi_{g,C}}\ar[d]_{\delta_{C,C'}}&\Escr_{gC}\ar[d]^{\delta_{gC,gC'}}\\
\Escr_{C'}\ar[r]_{\phi_{g,C'}}&\Escr_{gC'}
}
\end{equation}
so that the implied natural isomorphism $\delta_{gC,gC'}\phi_{g,C}\cong \phi_{g,C'}\delta_{C,C'}$ should again satisfy a number of obvious compatibilities. Note that
by adjointness we automatically have similar diagrams as \eqref{eq:deltadiagram} for the $\gamma$'s.
An $\Hscr$-schober
equipped with a $\Gscr$-action will be called a \emph{$\Gscr$-equivariant $\Hscr$-schober}.
\begin{remark}
\label{rem:stack}
Below we think of a $\Gscr$-equivariant $\Hscr$-schober as a perverse schober on the stack $V_{\CC}/\Gscr$. This is especially interesting if $\Gscr$
acts freely and discretely on~$V$ as then $V_{\CC}/\Gscr$ is a complex manifold.
\end{remark}
\subsection{\boldmath $\Hscr$-schobers versus spherical pairs}
\label{sec:spherical}
$\Hscr$-schobers form a satisfactory categorification of perverse
sheaves on hyperplane arrangements, but it is important to observe
that it is not the strongest possible notion. To see this it is
instructive to consider the example of the disc, i.e.\ the complex hyperplane
arrangement $(\CC,0)$ with the real version $\Hscr=(\RR,0)$. In that case Kapranov and Schechtman identify
 perverse schobers  with
spherical functors (see (1.8) in \cite{KapranovSchechtmanSchobers}) and it
was observed by Halpern-Leistner and Shipman in \cite[Theorem 3.15]{halpern2016autoequivalences} that spherical
functors are essentially the same as 
 \emph{$4$-periodic semi-orthogonal decompositions}
\begin{equation}\label{eq:4sod}
\Escr_0=\la\Dscr_-,\Escr_-\ra=\la\Escr_-,\Dscr_+\ra=\la\Dscr_+,\Escr_+\ra=\la\Escr_+,\Dscr_-\ra.
\end{equation}
In other words this data is completely determined by an admissible subcategory $\Escr_+\subset \Escr_0$ 
whose mutation helix is 4-periodic
\[
\ldots,\Dscr_+,\Escr_+,\Dscr_-,\Escr_-,\Dscr_+,\Escr_+,\ldots
\]
\begin{observation}
\label{obs:stronger}
It seems that in actual examples,  notably the ones  that we  consider in this paper, 
$(\Escr_+,\Escr_0,\Escr_-):=(\Escr_{C_1}, \Escr_{C},\Escr_{C_2})$ 
for $C_1,C_2$ as in (I) and $C=C_1\wedge C_2$ 
often gives rise to
 a 4-periodic semi-orthogonal decompositions but this does not
follow from the definition of an $\Hscr$-schober.
\end{observation} We given a concrete example below (see Example \ref{ex:flober})
but first we note that the notion of a $4$-periodic semi-orthogonal decompositions can
be generalized to that of a \emph{spherical pair}.
\begin{definition}\cite[Definition 3.5]{KapranovSchechtmanSchobers}
Let $\Escr_0$ be a triangulated category with admissible subcategories $\Escr_\pm$ and semi-orthogonal decompositions 
\[
\la \Dscr_-,\Escr_-\ra =\Escr_0=\la \Dscr_+,\Escr_+\ra.
\] 
Let $i_\pm:\Dscr_+\to\Escr_0$, $\delta_\pm:\Escr_+\to\Escr_0$ be inclusions and let ${}^*i_\pm:\Escr_0\to \Dscr_\pm$, $\delta^*_\pm:\Escr_\pm\to\Escr_0$ be respective left and right adjoints. 
Then $\Escr_\pm\subset \Escr_0$ is a {\em spherical pair} if ${}^*i_\mp i_\pm:\Dscr_\pm\to\Dscr_\mp$, $\delta^*_\mp\delta_\pm:\Escr_\pm\to \Escr_\mp$ are equivalences. 
\end{definition}
\begin{proposition}\cite[Proposition 3.7]{KapranovSchechtmanSchobers}
If $\Escr_\pm\subseteq \Escr_0$ is a spherical pair, then ${}^*i_\mp\delta_\pm:\Escr_\pm\to \Dscr_\mp$ is a spherical functor.\footnote{For the notion of spherical functor to make sense one must
 assume that the triangulated categories are suitably enhanced, see \cite[Appendix A]{KapranovSchechtmanSchobers}.}
\end{proposition}
The fact that 4-periodic semi-orthogonal decompositions yield spherical pairs is \cite[Proposition B.3]{BondalBodzenta}.  To distinguish the two notions
we refer to the former as a \emph{mutation spherical pair}. 
\begin{remark}Note that the notion of spherical pair is strictly more
general. Indeed for any semi-orthogonal decomposition $\Escr_0=\langle \Escr^\perp_{+},\Escr_+\rangle$ we have a \emph{trivial spherical pair} given by $(\Escr_+,\Escr_+)$
which is a mutation spherical pair if and only if $\Escr_0=\Escr^\perp_+\oplus \Escr_+$.
\end{remark}
\begin{remark} If we have a semi-orthogonal decomposition  $\Escr_0=\langle \Escr^\perp_{+},\Escr_+\rangle$ such that $\Escr_0$ has a Serre functor $S$ then the mutation helix is
\[
\ldots,\Escr^\perp_+,\Escr_+,S^{-1}(\Escr_+^\perp),S^{-1}(\Escr_+),S^{-2}(\Escr^\perp_+),S^{-2}(\Escr_+),\ldots
\]
In other words we obtain a mutation spherical pair if and only if the semi-orthogonal decomposition is preserved by $S^2$ and in that case $\Escr_-=S(\Escr_+)$, $\Escr^\perp_-=S(\Escr_+^\perp)$.

This observation holds true if $\Escr_0$ is linear over a commutative noetherian ring possessing a dualizing complex and $S$ is the corresponding relative Serre functor (if it exists).
\end{remark}

\begin{example}[\protect{\cite[Example 1.5]{BondalKapranovSchechtman},\cite{BondalBodzenta}}]
\label{ex:flober}
  Consider the example where $Z$ is the three dimensional quadratic
  cone, let $p_{\pm}:X_{\pm}\r Z$ be its two crepant resolutions
  and let $X_0=X_+\times_Z X_-$. Let $\Hscr=(\RR,0)$ (so that $\Hscr_\CC=(\CC,0)$). The corresponding cell complex is $\Cscr=(\RR_{>0},0,\RR_{<0})$.

 Put
  $\delta_{\pm}:=Lp^\ast_{\pm}:D(X_{\pm})\r D(X_0)$,   $\gamma_{\pm}:=Rp_{\ast,\pm}:D(X_0)\r D(X_\pm)$, $(\Escr_{\RR_{>0}},\Escr_0,\allowbreak\Escr_{\RR_{<0}}):=(D(X_+),D(X_0),D(X_-))$.
Then the data $(\Escr,\delta,\gamma)$ forms an $\Hscr$-schober\footnote{This
$\Hscr$-schober is also called a \emph{flober} since $X_+\r X_-$ is a flop.} \emph{but} it is easy to
verify that it is \emph{not a spherical pair} and hence it is certainly not a mutation
spherical pair. On the other hand, it was shown in \cite{BondalBodzenta} (see also \cite{donovan2019mirror}) how to fix this. One should replace $D(X_0)$ by $D(X_0)/\Kscr$ where $\Kscr$ is the common kernel 
of $Rp_{\ast,\pm}$. After doing so one really obtains a mutation spherical pair.
\end{example}

\section{The SKMS associated to a quasi-symmetric representation}\label{sec:SKMSqs}
\emph{From now on we assume, unless otherwise specified, that $W$ is a quasi-symmetric $G$-representation (see \cite[\S1.6]{SVdB} and \S\ref{sec:notation}) and that the generic $T$-stabilizer is finite.}
We recall hyperplane arrangements constructed in \cite{HLSam} .  Let
$(\beta_i)_i$ be the $T$-weights of $W$.  
Following \cite{SVdB3,SVdB}, 
we introduce 
$\overline{\Sigma}=\{\sum_ia_i\beta_i\mid a_i\in [-1,0]\}\subseteq X(T)_\RR$ and
\[
\Delta_0=(1/2)\overline{\Sigma}.
\] 
The hypothesis on the generic $T$-stabilizer implies that $\overline{\Sigma}$ is full dimensional.
Set
$\Delta=-\overline{\rho}+\Delta_0$, and denote by $(H_i)_i$ the
affine hyperplanes in $X(T)_\RR$ spanned by facets of $\Delta$.
Put
$\tilde\Hscr=\bigcup_i (-H_i+X(T))$ and
\[\Hscr=\bigcup_i (-H_i+X(T))\cap X(T)_\RR^\Wscr.\]
Let $(\Cscr,\leq)$ be the
poset of faces in $X(T)_\RR^{\Wscr}$ corresponding to $\Hscr$. Similarly let $(\tilde{\Cscr},\leq)$ 
be the poset of faces in $X(T)_{\RR}$ corresponding
to $\tilde{\Hscr}$. Note that $\Hscr$ is clearly invariant under translation by $X(T)^\Wscr$.

\begin{definition}[\protect{\cite[\S6]{HLSam}, also \cite[Remark before \S1.2]{Donovan}}] \label{def:SKMS}The SKMS associated to the data $(G,W)$ is 
\[
\left.\left(X(T)^\Wscr_{\CC}\setminus \bigcup_{H\in \Hscr} H\otimes\CC\right)\right/X(T)^\Wscr.
\]
\end{definition}
\section{Perverse schobers on the SKMS associated to a quasi-symmetric representation}\label{sec:SKMSschober}
\subsection{Main result}
\label{sec:mainth}
Let $\Hscr$, $\Cscr$ be as in \S\ref{sec:SKMSqs}.
In this section we construct an $X(T)^\Wscr$ invariant $\Hscr$-schober
on $\Cscr$. Via Definition \ref{def:SKMS} this may then be thought of as a perverse schober on the partial compactification $X(T)_\CC^\Wscr/X(T)^\Wscr$ of the SKMS of a quasi-symmetric representation.

We now give the construction of the $\Hscr$-schober.
For each $C\in \Cscr$ choose $\deltaC_C\in C$. 
Then
\begin{equation*}
\label{eq:lattice}
(\deltaC_C+\Delta)\cap X(T)
\end{equation*}
is independent of $\deltaC_C\in C$ by Lemma \ref{lem:lema33} below.
We put
\begin{align*}
\Delta_C&=\deltaC_C+\Delta,\\
\Lscr_C&=\Delta_C\cap X(T)^+.
\end{align*}
Obviously $\Lscr_C$ also does not depend on $\deltaC_C\in C$. By Lemma \ref{lem:lema33} below, $C'\le C$ implies $\Lscr_C\subseteq \Lscr_{C'}$.

Put $X=\Spec \Sym(W)=W^\ast$ and define
\begin{align*}
P_\chi&=V(\chi)\otimes_k \Oscr_X\in D(X/G),\\
P_C&=\bigoplus_{\chi\in \Lscr_C} P_\chi,\\
\Escr_C&=\la P_C\ra \subset D(X/G).
\end{align*}
We let $\Escr^c_C$ be the full subcategory of $\Escr_C$ consisting of compact objects. 
For $C'\leq C$ let  
\[\delta_{CC'}:\Escr_{C}\hookrightarrow \Escr_{C'}\]
  be the inclusion.  Clearly $\delta_{CC'}$ preserves compact
objects. By Lemma \ref{lem:adjoints} below $\delta_{CC'}$ has a right adjoint $\gamma_{C'C}$ which also preserves compact objects. Finally  for $\chi\in X(G)=X(T)^\Wscr$ we define
\[
\phi_{\chi,C}:\Escr_{C}\r \Escr_{\chi+C}:\Mscr\mapsto \chi\otimes \Mscr.
\]
The following is our main result.
\begin{proposition}\label{prop:schober}
The data $((\Escr_C)_{C},(\gamma_{C'C})_{C'C},(\delta_{CC'})_{C'C})$  with the $X(T)^\Wscr$-action $(\phi_{\chi,C})_{\chi,C}$ defines an $X(T)^\Wscr$-equivariant $\Hscr$-schober on $X(T)_\CC^\Wscr$.

Moreover, 
\begin{enumerate}
\item For collinear $C_1,C,C_2$ with $C<C_1,C_2$, $(\Escr_{C_1},\Escr_{C_2})$ forms a mutation spherical pair in $\Escr_C$.
\item The same results are true with $\Escr^c_C$ replacing $\Escr_C$.
\end{enumerate}
\end{proposition}
\begin{remark} Denote by $X^s$ (resp.\ $X^u$) the stable (resp.\ non-stable) locus of $X$
and by $X^{ss,\chi}$ the
  semi-stable locus of $X$ for a linearization of the
  $G$-action by $\chi\otimes_k \Oscr_X$ for some $\chi\in X(G)=X(T)^\Wscr$. Assume that
  $\emptyset \neq X^s=X^{ss,\chi}$ and that $C$ is maximal. Then 
by \cite[Theorem 3.2]{HLSam}
 the restriction
  $\Pscr_C$ of $P_C$ to $X^{ss,\chi}$ is a tiling bundle on $X^{ss,\chi}/G$ and hence
  \[
\Escr_C\cong D(X^{ss,\chi}/G).
\]
As a corollary one obtains that all $\Escr_C$ (for maximal $C$) are in fact derived equivalent. This is also
a consequence of the more general Proposition \ref{prop:schober}. 

 If $W$ is moreover \emph{generic} (see \cite[Definition 1.3.4]{SVdB})  
then by a suitable modification of \cite[Theorem 1.6.3]{SVdB} it follows that $\Lambda_{C}:=\End_{X/G}(P_C)$ is a so-called \emph{non-commutative crepant resolution} of
  $k[X]^G$.
\end{remark}
\subsection{Properties of \boldmath $\Hscr$}
The following trivial lemma was used. 
\begin{lemma}\label{lem:lema33}
Let $E$ be vector space over $\RR$ and let $\Delta\subset E$ be a full dimensional convex polytope.
Let $S$ be a discrete subset of $E$.
Let $(H_i)_{i\in I}$ be the  set of hyperplanes spanned by the facets of $\Delta$.
Let $\Hscr$ be the affine hyperplane
arrangement given by $-H_i+s$, $s\in S$, $i\in I$ and let $E=\coprod_{C\in \Cscr}C$ be the corresponding cell decomposition
of $E$.  Then the following holds:
\begin{enumerate}
\item \label{it:1}
$
(\deltaC+\Delta)\cap S
$
does not depend on $\deltaC\in C$ for $C\in \Cscr$. Write $S_C:=(\deltaC+\Delta)\cap S$. 
\item \label{it:2}
If $C'\le C$ then $S_{C}\subseteq S_{C'}$.
\end{enumerate}
\end{lemma}
\begin{proof}
Assume
\[
\Delta=\{x\in E\mid \phi_i(x)\le 0,i\in I\}
\]
where $\phi_i(x)=0$ are the equations for $H_i$. Then the equations for
the hyperplanes in $\Hscr$ are $\phi_i(s-x)=0$ for $i\in I$, $s\in S$.
For $a\in \RR$ let $\epsilon(a)=-1,0,1$ depending on whether $a<0$, $a=0$ or $a>0$. Then
the functions $\epsilon_{i,s}(x):=\epsilon(\phi_i(s-x))$ are constant on $C$ (and in fact they define $\Cscr$).
Denote their values by $\epsilon_{i,s}(C)$.

Now assume $\deltaC\in C$. Then
\begin{align*}
(\deltaC+\Delta)\cap S&=\{s\in S\mid \phi_i(s-\deltaC)\le 0,i\in I\}\\
&=\{s\in S\mid  \epsilon_{i,s}(\deltaC)\in \{-1,0\},i\in I\}\\
&=\{s\in S\mid  \epsilon_{i,s}(C)\in \{-1,0\},i\in I\}
\end{align*}
which is clearly independent of $\deltaC\in C$.

If $C'$ is as in the statement of the lemma then we have $\epsilon_{s,i}(C')=t_{s,i}\epsilon_{s,i}(C)$ for some $t_{s,i}\in \{0,1\}$.
This implies the inclusion $(\deltaC+\Delta)\cap S\subseteq (\deltaC'+\Delta)\cap S$.
\end{proof}
Let the notation be as in \S\ref{sec:SKMSqs}, \S\ref{sec:mainth}.
For use below we give some more trivial lemmas.
\begin{lemma} 
\label{eq:deltainvariance} 
Let $\sigma\in \pm \Wscr$. Then $\sigma(\Delta)=\xi_\sigma+\Delta$ for $\xi_\sigma\in X(T)$. If $\sigma=-w_0$ then $\xi_\sigma=0$.
\end{lemma}
\begin{proof}
We have $\Delta=-\overline{\rho}+\Delta_0$. $\Delta_0$ is $\Wscr$-invariant and the fact that $W$ is unimodular also implies that $\Delta_0$ is invariant under $x\mapsto -x$. So it  is
sufficient that $\sigma(\overline{\rho})-\overline{\rho}\in X(T)$. This is standard (it is enough to consider the case $\sigma(x)=-x$ which is trivial and the case that $\sigma$ is
a simple reflection $s_\alpha$ in which case we have $\sigma(\overline{\rho})-\overline{\rho}=-\alpha$). The second claim follows from the fact that $-w_0\overline{\rho}=\overline{\rho}$.
\end{proof}
\begin{lemma} \label{rem:Hinvariance}
$\tilde{\Hscr}$ is invariant under $\pm \Wscr$. $\Hscr$ is stable under $x\mapsto -x$.
The same holds for $\tilde{\Cscr}$ and $\Cscr$.
\end{lemma} 
\begin{proof} The first claim follows from Lemma \ref{eq:deltainvariance} and the fact that $\tilde{\Hscr}$ consists of the supporting hyperplanes for $-\Delta+\chi$ for $\chi\in X(T)$.
The rest of the claims are clear.
\end{proof}
\subsection{Adjoints}
We have
\begin{equation*}
\label{eq:identification}
\begin{aligned}
\Escr_C&\cong D(\Lambda^\circ_C),\\
\Escr^c_C&\cong \Perf(\Lambda^\circ_C),
\end{aligned}
\end{equation*}
where
$
\Lambda_C=\End_{X/G}(P_C).
$
\begin{theorem}\protect{\cite[Theorem 1.6.1]{SVdB}}
\label{th:global}
$\Lambda_C$ has finite global dimension.
\end{theorem}

\begin{lemma} \label{lem:adjoints} $\delta_{CC'}$ has  a right 
denoted
by $\gamma_{C'C}$ 
which sends $\Escr_{C'}^c$ to $\Escr_{C}^c$. We have
\begin{equation}
\label{eq:gamma}
\gamma_{C'C}=\RHom_{X/G}(P_C,-)\otimes_{\End_{X/G}(P_C)} P_C.
\end{equation}
\end{lemma}

\begin{proof}
  It is easy to see that \eqref{eq:gamma} is the right adjoint to
  $\delta_{CC'}$. When we evaluate it on $P_\chi$ for $\chi\in C'$ then we see 
by Theorem \ref{th:global} that the image lies in $\Escr^c_C$.
\end{proof}

\subsection{Duality}
Below we use the autoduality functor 
\[\Dscr:D^b(\coh(X/G))\r D^b(\coh(X/G))^\circ,\]
\[
\Dscr=\uRHom_{X/G}(-,\Oscr_X).
\]
Obviously $\Dscr\circ \Dscr\cong\Id$.
\begin{lemma}
\label{lem:duality}
We have $
\Dscr(\Escr^c_C)= \Escr^c_{-C}$.
\end{lemma}
\begin{proof} Let $\chi \in (\deltaC_C+\Delta)\cap X(T)^+$. Then $\Dscr(P_\chi)=P_{-w_0\chi}$. 
By Lemmas \ref{eq:deltainvariance},\ref{rem:Hinvariance}  we have 
$-w_0\chi\in -\deltaC_C+\Delta=\deltaC_{-C}+\Delta=\Delta_{-C}$ and hence $-w_0\chi\in \Lscr_{-C}$.
\end{proof}

\medskip

Besides the objects $P_\chi$ in $D(X/G)$ introduced in \S\ref{sec:mainth} we will also be interested in the objects\footnote{See \S\ref{sec:notation} for our use of $\Ind^G_B(-)$.} 
$
\RInd^G_B(\chi\otimes \Oscr_{X^{\lambda,+}})
$ 
for $\chi \in X(T)^+$, $\lambda\in Y(T)^-_{\RR}$. 
\begin{lemma}\cite[Proof of Lemma 11.2.1]{SVdB} \label{lem:workhorse} \label{lem:weyman1}
The cohomologies  
of $\RInd^G_B(\chi\otimes \Oscr_{X^{\lambda,+}})$ are (as $G$-representations) direct sums of $V(\mu)$ with
$\langle \lambda,\mu\rangle\le \langle \lambda,\chi\rangle$.
\end{lemma}
\begin{lemma}\cite[Theorems (5.1.2), (5.1.4), \S5.4]{weyman2003cohomology}
Let $\beta_\lambda=\sum_{\la\lambda,\beta_i\ra>0}\beta_i$ and let $d_\lambda=|\{i\mid \la\lambda,\beta_i\ra>0\}|-\dim(G/B)$. Then
\begin{equation*}\label{eq:dual0}
\Dscr(\RInd_{B}^G(\chi\otimes \Oscr_{X^{\lambda,+}}))=
\RInd_{B}^G\left((-2\overline{\rho}-\chi-\beta_\lambda)\otimes \Oscr_{X^{\lambda,+}}\right)[-d_\lambda].
\end{equation*}
\end{lemma}

\medskip

For $C<C'\in \Cscr$ we put
\[
\Escr_{C,C'}=\Escr^\perp_{C'}\cap \Escr_C.
\]
For specific choices of $\deltaC_C\in C$, $\deltaC_{C'}\in C'$ we put  $\varepsilon=\deltaC_{C'}-\deltaC_{C}$.  
In the notation of Corollary \ref{cor:maincor} with $\Pi=\Delta_C$ we have
$\Escr=\Escr_C$, $\Escr_\varepsilon=\Escr_{C'}$ 
(using Lemma \ref{lem:lema33}) and hence by \eqref{eq:sod}, $\Escr_{C,C'}=\overline{\Escr}_\varepsilon$.
Concretely we obtain from Corollary \ref{cor:maincor}:
\begin{multline*}
\Escr_{C,C'}=\left\la \RInd^G_B(\chi\otimes \Oscr_{X^{\lambda,+}})\mid \chi\in \Lscr_{C}\setminus \Lscr_{C'}, \lambda\in Y(T)^-_\RR\right.\\
\left.\text{ such that }\langle \lambda,\zeta\rangle\ge\langle \lambda,\chi\rangle {\text  { for }}\zeta\in \Delta_C, \langle \lambda,\varepsilon\rangle>0\vphantom{\RInd^G_B}\right\ra.
\end{multline*}
In particular $\Escr_{C,C'}$ is generated by compact objects in $D(X/G)$.

\medskip 

The next lemma elucidates how $\Escr_{C,C'}$ interacts
with duality.
\begin{lemma}\label{lem:ddual} Let $C_1,C,C_2$ be collinear faces in $\Cscr$ such that 
$C<C_1,C_2$ ($C_1$ and $C_2$ determine each other). Then 
\[
\Dscr(\Escr^c_{C,C_1})= \Escr^c_{-C,-C_2}.
\]
\end{lemma}
\begin{proof}
We use the notation of Corollary \ref{cor:maincor}, but to indicate the context we write $F_{C,\lambda}=F_\lambda$. 
We may assume $\deltaC_{C_2}+\deltaC_{C_1}=2\deltaC_C$.
Set 
\[
\varepsilon=\deltaC_{C_1}-\deltaC_{C}=-\delta_{C_2}+\deltaC_{C}=\deltaC_{-C_2}-\deltaC_{-C}.
\]
By definition, $\Escr_{C,C_1}$ (resp. $\Escr_{-C,-C_2}$) is generated by
$\RInd_{B}^G(\chi\otimes \Oscr_{X^{\lambda,+}})$ for 
$\chi\in \overline{F}_{C,\lambda}$ (resp. $\chi\in \overline{F}_{-C,\lambda}$) with $\lambda\in Y(T)_\RR^-$,  $\la \lambda,\varepsilon \ra>0$. 
By  Lemma \ref{lem:weyman1}, we therefore need to check that for $\lambda\in Y(T)_\RR^-$ with $\la\lambda,\varepsilon \ra>0$ we have
\[
\{-2\overline{\rho}-\chi-\beta_\lambda\mid \chi\in \overline{F}_{C,\lambda}\}=\{\chi\in \overline{F}_{-C,\lambda}\}.
\]
By duality, it is enough to prove one inclusion. 
Assume $\chi\in \overline{F}_{C,\lambda}$. 
Using \eqref{eq:closure1} below we can write 
\[
\chi=\deltaC_C-\overline{\rho}-(1/2)\beta_\lambda+\sum_{\la\lambda,\beta_i\ra=0}c_i\beta_i
\]
with $c_i\in [-1/2,0]$. Hence 
\begin{align*}
-2\overline{\rho}-\chi-\beta_\lambda&=-\deltaC_C-\overline{\rho}-(1/2)\beta_\lambda-\sum_{\la\lambda,\beta_i\ra=0}c_i\beta_i\\
&=\deltaC_{-C}-\overline{\rho}-(1/2)\beta_\lambda+ \sum_{\la\lambda,\beta_i\ra=0}c'_i\beta_i
\end{align*}
for some $c'_i\in [-1/2,0]$ using quasi-symmetry. By Lemma \ref{lem:closure} below, $-2\overline{\rho}-\chi-\beta_\lambda\in \overline{F}_{-C,\lambda}$ as desired.
\end{proof}

\subsection{Mutation spherical pairs}
\begin{proposition}\label{prop:mutation} Let $C_1,C,C_2$ be collinear faces in $\Cscr$ such that 
$C<C_1,C_2$.
There are semi-orthogonal decompositions 
\[
\Escr_C=\langle \Escr_{C,C_1},\Escr_{C_1}\rangle=\langle \Escr_{C_1},\Escr_{C,C_2}\rangle=\langle \Escr_{C,C_2},\Escr_{C_2}\rangle=\langle \Escr_{C_2},\Escr_{C,C_1}\rangle,
\]
\[
\Escr^c_C=\langle \Escr^c_{C,C_1},\Escr^c_{C_1}\rangle=\langle \Escr^c_{C_1},\Escr^c_{C,C_2}\rangle=\langle \Escr^c_{C,C_2},\Escr^c_{C_2}\rangle=\langle \Escr^c_{C_2},\Escr^c_{C,C_1}\rangle.
\] 
In other words, $\Escr_{C_1},\Escr_{C_2}\subseteq \Escr_C$,  $\Escr^c_{C_1},\Escr^c_{C_2}\subseteq \Escr^c_C$ are mutation spherical pairs. 
\end{proposition}

\begin{proof}
In our current notation
Corollary \ref{cor:maincor} implies that there is a semi-orthogonal decomposition
\[
\Escr_C=\langle \Escr_{C,C_1},\Escr_{C_1}\rangle.
\]
By inspecting the proof, there is a semi-orthogonal decomposition
\[
\Escr^c_C=\langle \Escr^c_{C,C_1},\Escr^c_{C_1}\rangle.
\]
By using duality in Lemmas \ref{lem:duality}, \ref{lem:ddual} (replacing $C$ by $-C$, $C_1$ by $-C_1$, $C_2$ by $-C_2$), and by  interchanging $C_1$ and $C_2$  one respectively obtains 
three more semi-orthogonal decompositions
\[
\Escr^c_C=\langle \Escr^c_{C_1},\Escr^c_{C,C_2}\rangle=\langle \Escr^c_{C,C_2},\Escr^c_{C_2}\rangle=\langle \Escr^c_{C_2},\Escr^c_{C,C_1}\rangle.
\] 
From these one easily obtains corresponding semi-orthogonal decompositions with the $(-)^c$ omitted:
\[
\Escr_C=\langle \Escr_{C_1},\Escr_{C,C_2}\rangle=\langle \Escr_{C,C_2},\Escr_{C_2}\rangle=\langle \Escr_{C_2},\Escr_{C,C_1}\rangle.\qedhere
\]
\end{proof}
\subsection{Proof of Proposition \ref{prop:schober}}
The $X(T)^\Wscr$-equivariance is obvious.
The first part of Definition \ref{def:schober} is satisfied by the definition of $\delta_{C'C}$, $\gamma_{CC'}$.  We need to verify the conditions $(M),(I),(T)$. 
\begin{enumerate}
\item[(M)] This follows from the fact that $(\delta_{CC'},\gamma_{C'C})$ is
an adjoint pair and $\delta_{CC'}$ is fully faithful.
\item[(I)]\label{I} 
Put $C=C_1\wedge C_2$. The property is a consequence of Proposition \ref{prop:mutation}, which implies that $\phi_{C_1C_2}$ is the mutation functor. 
\item[(T)]
Since $C_1,C_2,C_3$ are collinear we may assume
\begin{equation}
\label{eq:interval}
\xi_{C_2}\in [\xi_{C_1},\xi_{C_3}].
\end{equation} 
The interval $[\xi_{C_1},\xi_{C_3}]$  passes through different faces of $\Cscr$ (all sharing $C_0$)
and by considering those we reduce formally to the case where $C_1$, $C_2$ are \emph{neighboring}, i.e. $C_1<C_2$ or $C_2<C_1$. So below assume
we are in this situation. 

We claim that in case $C_1<C_2$ we have
\begin{equation}
\label{perp}
\Escr_{C_1,C_2}\subset \Escr_{C_3}^\perp.
\end{equation}
Let $\RInd^G_B(\chi\otimes \Oscr_{X^{\lambda,+}})$ be a defining generator of $\Escr_{C_1,C_2}$. Then 
$\la \lambda,-\ra \ge\la\lambda,\chi\ra$ on $\Delta_{C_1}$ and $\langle \lambda,\xi_{C_2}-\xi_{C_1}\rangle>0$. By \eqref{eq:interval} the latter implies $\langle \lambda,\xi_{C_3}-\xi_{C_1}\rangle>0$. 
Let $P_\mu$ be a defining generator of $\Escr_{C_3}$. Then $\mu-\xi_{C_3}+\xi_{C_1}\in \Delta_{C_1}$ and hence $\la \lambda,\mu-\xi_{C_3}+\xi_{C_1}\ra \ge\la\lambda,\chi\ra$. It follows
$\la \lambda ,\mu\ra >\la \lambda,\chi\ra$. We finish the proof of the claim by invoking Lemma \ref{lem:weyman1}.

\medskip

Now it remains to prove that the canonical natural transformation 
\[
\phi_{C_2C_3}\phi_{C_1C_2}\r \phi_{C_1C_3}
\] is an isomorphism and it is sufficient do this after evaluation on
a generator~$P_{\chi}$ with $\chi \in \Lscr_{C_1}$. In other words we have to prove
\begin{equation}
\label{eq:pchi}
\gamma_{C_0C_3}\gamma_{C_0C_2}(P_\chi)=\gamma_{C_0C_3}(P_\chi)
\end{equation}
(we have not written the $\delta$'s). If $C_2<C_1$ then $\Escr_{C_1}\subset \Escr_{C_2}$ and
hence $\gamma_{C_0C_2}(P_\chi)=P_\chi$ and so there is nothing prove. 
So assume $C_1<C_2$. In that case  $\gamma_{C_0C_2}P_\chi=\gamma_{C_1C_2}P_\chi$ and by Proposition \ref{prop:mutation}
\begin{equation}\label{eq:ec1c2}
\cone(\gamma_{C_1C_2}P_\chi\r P_\chi)\in  \Escr_{C_1,C_2}.
\end{equation}
Then \eqref{eq:pchi} follows by applying $\gamma_{C_0C_3}$ to \eqref{eq:ec1c2} 
and invoking \eqref{perp}.
\end{enumerate}

(1) is a restatement of Proposition \ref{prop:mutation}. 
(2) follows from the main statement of the proposition (and Proposition \ref{prop:mutation}). 
\qedhere

\appendix
\section{Explicit semi-orthogonal decompositions}
The purpose of this appendix is to give a self-contained exposition on the results\footnote{In \cite{SVdB3} we are mainly concerned with producing semi-orthogonal decompositions
of $D(X/G)$ whose main part is a non-commutative resolution of $X\quot G$. Such a non-commutative resolution may be of the form $D(X^{ss,\chi}/G)$ but it does not have to be.
In particular in the linear case the theory in loc.\ cit.\ does not depend on the existence of non-trivial characters
to produce non-trivial results. So it also applies when $G$ is semi-simple.}
 from \cite{SVdB3} which were used above, and are somewhat dispersed in loc.cit. 

\subsection{Faces and fans}
\label{sec:someresults}
Let $E$ be a finite dimensional real vector space and let $\Pi$ be a full dimensional
polyhedron in $E$.
Let us say that $\lambda\in E^\ast$ defines a supporting half plane for a face $F$ of $\Pi$ if
there is a supporting half space for $F$ of the form $\{x\mid \langle\lambda,x\rangle\ge u\}$, i.e.
$\langle \lambda,f\rangle\ge u$ for $f\in \Pi$ with equality if and only if $f\in \overline{F}$.
Clearly $F$ and $u$ are determined by~$\lambda$ and we write $F_\lambda=F$, $u_\lambda=u$. For $f\in \Pi$ we put
\[
\sigma_f=\{\lambda\in E^\ast\mid f\in F_\lambda\}.
\]
This is an open polyhedral cone with unique vertex $0$.
For use below we note
\begin{equation}
\label{eq:barsigma}
\overline{\sigma}_f=\{\lambda\in E^\ast\mid f\in \overline{F}_\lambda\}.
\end{equation}
It is easy to see that $\sigma_f$ only depends on the face $F$ that $f$ belongs to. We write $\sigma_F=\sigma_f$.
\begin{proposition}\cite[Theorem 2.3.2, Propositions 2.3.8, 2.3.7]{CoxLittleSchenck}
\label{prop:CLS}
\mbox{}
\begin{enumerate}
\item
$\Sigma_\Pi:=(\overline{\sigma_{F}})_F$ is a  fan\footnote{$\Sigma_\Pi$ is called the ``normal fan of $\Pi$''.} in $E^\ast$ such that $\coprod_F\sigma_F=E^\ast$.
\item Let $\Fscr(\Delta)$ be the set of (open) faces of $\Delta$, ordered by $F\le F'$ iff $F\subset \overline{F'}$. Then the map
$F\mapsto \overline{\sigma_{F}}$ is an order inverting isomorphism between $\Fscr(\Pi)$
and $\Sigma_{\Pi}$.
\item
The function is $\lambda\mapsto u_\lambda$ is continuous and piecewise linear on $\Sigma_\Pi$.
\end{enumerate}{}
\end{proposition}
Now assume that $E$ is equipped with a positive definite inner product $(-,-)$.
This induces an identification $E\cong E^*$ and a positive definite inner
product on $E^\ast$, also denoted by $(-,-)$.

We fix $0\neq \varepsilon\in X(T)_{\RR}$. For $0\neq \lambda\in E^\ast$ put
\[
q(\lambda)=\frac{\langle\lambda,\varepsilon\rangle}{\|\lambda\|}.
\]
We put 
\[
H_\varepsilon=\{\lambda\in E^\ast\mid \langle \lambda,\varepsilon\rangle >0\}
\]
and set
\[
\tau_{f}=\sigma_f\cap H_\varepsilon.
\]
$\tau_f$ is an (open) polyhedral cone in $E^\ast$
 if it is non-empty. We have
\begin{equation*}
\label{eq:tauempty}
\tau_f\neq\emptyset \iff \forall \kappa>0:f\not\in\kappa\varepsilon+\Pi.
\end{equation*}
If $\tau_f\neq \emptyset$ then
it is easy to see (see Corollary \ref{cor:maximum} below) that $q{\mid} (\overline{\tau}_f-\{0\})$
takes its maximum values on a unique ray $\RR_{>0}\lambda_f$. We put
$q_f=q(\lambda_f)$. If $\tau_f=\emptyset$ then we put $q_f=-\infty$.
\subsection{Main results}
We recall that in \cite[\S11.2, Proof of Lemma 11.2.1]{SVdB} certain bounded complexes 
$C_{\lambda,\chi}$ where constructed for $\chi\in X(T)^+$, $\lambda\in Y(T)^-_{\RR}$, 
 computing 
$\RInd^G_B(\chi\otimes \Oscr_{X^{\lambda,+}})$. 

The terms in $C_{\lambda,\chi}$ are of the form $P_\zeta$
 with
\begin{equation}
\label{eq:sigmalist}
\zeta=(\chi+\beta_{i_1}+\cdots+\beta_{i_{p}})^+
\end{equation}
where $\{i_1,\ldots,i_{p}\}\subset \{1,\ldots,d\}$, $d=\dim W$, $i_j\neq i_{j'}$ for $j\neq j'$, $\la \lambda,\beta_{i_j}\ra>0$. In particular
$P_{\chi}$ occurs once and the canonical morphism 
\[P_\chi=\RInd^G_B(\chi \otimes \Oscr_X)\r \RInd^G_B(\chi\otimes \Oscr_{X^{\lambda,+}})\]
is represented by a morphism of complexes 
$P_\chi\r C_{\lambda,\chi}$
whose cone is in $\langle (P_{\zeta})_{\zeta,p\neq 0}\rangle$ where $\zeta$ is as in \eqref{eq:sigmalist}. 

The following is our main combinatorial result about the complexes $C_{\lambda,\chi}$.
\begin{proposition} \label{prop:mainprop} Assume that $W$ is quasi-symmetric and the generic $T$-stabilizer is finite (so that $\Pi$ is full dimensional).
Put
  $\Pi=\deltaC-\overline{\rho}+\Delta_0\subset X(T)_\RR:=E$ with
  $\deltaC\in X(T)_\RR^\Wscr$ and fix 
  $0\neq\varepsilon\in X(T)^\Wscr_{\RR}$ as above. Choose $(-,-)$ to be $\Wscr$-invariant.
 Let
  $\chi\in \Pi\cap X(T)^+$ be such that
  $\forall\kappa>0:\chi\not\in \kappa\varepsilon +\Pi$
  (i.e. $q_\chi\neq-\infty$). 
\begin{enumerate} 
\item \label{it:mainprop1} We have $\lambda_\chi\in
  Y(T)_\RR^-$.
\item \label{it:mainprop2}
For every $\lambda\in Y(T)_\RR^-\cap \overline{\sigma}_\chi$ 
the terms $P_\zeta$,  occurring in $C_{\lambda,\chi}$ (see \eqref{eq:sigmalist}) satisfy $\zeta\in \Pi$.
\item \label{it:mainprop3}
Moreover if $\lambda=\lambda_\chi$  and $\zeta\neq \chi$ 
then those terms satisfy in addition $q_\zeta<q_\chi$.
\end{enumerate}
\end{proposition}
For the benefit of the reader we give a self-contained proof of this combinatorial proposition in \S\ref{sec:selfcontained} below.
The ensuing corollary below may be deduced
from the results in \cite{SVdB3}   
and also from the  results in
\cite{HL,HLSam}. See Remarks \ref{rem:HLS},\ref{rem:sod} below.
\begin{corollary} \label{cor:maincor} Assume that $W$ is quasi-symmetric and the generic $T$-stabilizer is finite.
Let $\Pi$ be as in Proposition \ref{prop:mainprop}.
  Put
  $\Pi_\varepsilon=\bigcup_{\kappa>0} \Pi\cap
  (\kappa\varepsilon+\Pi)$.
  Put $\Lscr=\Pi\cap X(T)^+$, $\Lscr_\varepsilon=\Pi_{\varepsilon}\cap X(T)^+$ and
  let $\Escr=\langle P_\zeta\rangle_{\zeta\in \Lscr} $ and $\Escr_\varepsilon=\langle
P_\zeta\rangle_{\zeta \in \Lscr_\varepsilon}$. Let $\overline{\Escr}_\varepsilon$ be
the full subcategory of $D(X/G)$ spanned by $\RInd^G_B(\mu\otimes \Oscr_{X^{\lambda,+}})$
for $\mu\in \Lscr\setminus \Lscr_\varepsilon$ and 
$\lambda\in \overline{\sigma}_\mu\cap Y(T)^-_\RR$, $\langle\lambda,\varepsilon\rangle>0$. Then there is a semi-orthogonal decomposition
\begin{equation}
\label{eq:sod}
\Escr=\langle \overline{\Escr}_\varepsilon,\Escr_\varepsilon\rangle.
\end{equation}
\end{corollary}
\begin{proof} The fact that
  $\overline{\Escr}_\varepsilon\subset \Escr$ follows from Proposition
  \ref{prop:mainprop}\eqref{it:mainprop2} (using the above fact that $C_{\lambda,\chi}$ computes $\RInd^G_B(\chi\otimes \Oscr_{X^{\lambda,+}})$).  The fact that $\Escr$ is
  generated by $\Escr_\varepsilon$ and $\overline{\Escr}_\varepsilon$
  follows by repeatedly applying Proposition
  \ref{prop:mainprop}(\ref{it:mainprop1},\ref{it:mainprop3}). So we
  have to prove orthogonality; i.e. for $\zeta\in \Lscr_\varepsilon$
  and $\mu\in \Lscr\setminus \Lscr_\varepsilon$,
  $\lambda\in \overline{\sigma}_\mu\cap Y(T)^-_\RR$, $\langle\lambda,\varepsilon\rangle>0$ we must have
\[
\Hom_{X/G}(P_\zeta,\RInd^G_B(\mu\otimes \Oscr_{X^{\lambda,+}}))=0.
\]
This follows from Lemma \ref{lem:workhorse}, using the fact that we
have $\langle\lambda,\zeta\rangle>\langle\lambda,\mu\rangle$
(by the definition of $\Pi_\varepsilon$ there is some $\kappa>0$ such that $\zeta-\kappa\epsilon\in \Pi$, hence $\la \lambda,\zeta\ra
>\la \lambda,\zeta-\kappa\epsilon\ra\ge \la \lambda,\mu\ra$).
\end{proof}
\begin{remark} We have stated Corollary \ref{cor:maincor} in a way that is independent
of the choice of $(-,-)$. However we could also have defined $\overline{\Escr}_\varepsilon$ as
the full subcategory of $D(X/G)$ generated by  $\RInd^G_B(\mu\otimes \Oscr_{X^{\lambda_\mu,+}})$ for
$\mu\in \Lscr\setminus \Lscr_\varepsilon$. In this way it is easy to see
that $\overline{\Escr}_\varepsilon$ can be further decomposed according to the value of $q_\mu$.
\end{remark}
\begin{remark} 
\label{rem:HLS} In \cite{HL} (see also \cite{BFK}) Halpern-Leistner constructs under very general conditions an (infinite) semi-orthogonal decomposition
of $D(X/G)$ for a linearized quotient stack $X/G$ in terms of \emph{windows} 
(a concept introduced in \cite{SegalDonovan}). Using the windows description of $\Pi$ in
\cite[Definition 2.5, Lemma 2.8, \S3.1]{HLSam}, and the fact that the inclusion $\Pi_\varepsilon\subset \Pi$ is
obtained by replacing some of the closed intervals describing $\Pi$ by half-open intervals,
 one may view \eqref{eq:sod} as realizing, using explicit generating objects, a fragment 
of that semi-orthogonal decomposition in the case
that $X$ is a quasi-symmetric representation.
\end{remark}
\begin{remark} \label{rem:sod}
An alternative way of proving \eqref{eq:sod} is to replace $\Pi$ by a slightly
scaled and translated version $\delta'-\overline{\rho}+r\Delta_0$ for $r>1$ and $\delta'=\delta+\kappa\varepsilon$, $0<\kappa\ll 1$
chosen in  such a way
that $\Lscr=(\delta'-\overline{\rho}+r\Delta_0)\cap X(T)$, $\Lscr_\varepsilon=(\delta'-\overline{\rho}+\Delta_0)\cap X(T)$.
Then one may invoke the results of \cite[\S8]{SVdB3}. But as said above, we have preferred to give
a self-contained proof here.
\end{remark}
\subsection{Faces of zonotopes}
In this section we use the notations introduced in \S\ref{sec:someresults} but
we now consider the case that $\Pi$ is a zonotope
\[
\Pi=\sum_{i=1}^d [v_i,w_i]
\]
and ${v_i},w_i\in E$. It is well known that in this case
the  $F_\lambda$ introduced in \S\ref{sec:someresults}
have the following concrete description \cite[Appendix B]{SVdB}\label{lem:oneparam}
\begin{equation*}
\label{eq:Flambdadef}
F_\lambda=\sum_{\la\lambda,{w_i}\rangle>\la\lambda,v_i\ra}v_i+\sum_{\la\lambda,{w_i}\rangle<\la \lambda, v_i\ra }w_i+
\sum_{\la\lambda,{v_i}\rangle=\la \lambda, w_i\ra}]a_i,b_i[{v_i},
\end{equation*}
so that in particular
\begin{equation}
\label{eq:ulambda}
u_\lambda=\sum_{\la\lambda,{w_i}\rangle>\la\lambda,v_i\ra}\la \lambda,v_i\ra+\sum_{\la\lambda,{w_i}\rangle<\la \lambda, v_i\ra }\la \lambda,w_i.\ra
\end{equation}
We also have the following convenient characterization of $\overline{F}_\lambda$ and $\overline{\sigma}_f$.
\begin{lemma}
\label{lem:closure}
Assume $f=\sum_i r_i$, $r_i\in [v_i,w_i]$. Let $\lambda\in E^\ast$.
Then $f\in \overline{F}_\lambda$ (or equivalently $\lambda\in \overline{\sigma}_f$ by \eqref{eq:barsigma})
if and only if
\begin{equation}
\label{eq:closure1}
\begin{aligned}
\langle \lambda,{w_i}\rangle>\langle \lambda,{v_i}\rangle&\Longrightarrow r_i=v_i,\\
\langle \lambda,{w_i}\rangle<\langle \lambda,{w_i}\rangle&\Longrightarrow r_i=w_i.
\end{aligned}
\end{equation}
\end{lemma}
\begin{proof} By \eqref{eq:ulambda} we have $f\in \overline{F}_\lambda$ if and only if
\[
\langle \lambda,f\rangle
=\sum_{\la\lambda,{w_i}\rangle>\la\lambda,v_i\ra}\la \lambda,v_i\ra+\sum_{\la\lambda,{w_i}\rangle<\la \lambda, v_i\ra }\la \lambda,w_i\ra.
\]
On the other hand inspecting the inequality
\begin{equation*}
\langle \lambda,f\rangle=\sum_i \langle\lambda,{r_i}\rangle
\ge u_\lambda= \sum_{\la\lambda,{w_i}\rangle>\la\lambda,v_i\ra}\la \lambda,v_i\ra+\sum_{\la\lambda,{w_i}\rangle<\la \lambda, v_i\ra }\la \lambda,w_i\ra
\end{equation*}
we see that it is an equality if and only if
 \eqref{eq:closure1} is true.
\end{proof}
\subsection{Some convex geometry}
In this section we use the notations introduced in \S\ref{sec:someresults}. 
We remind
the reader of a trivial lemma. Let
$
\SS=\{\lambda\in E\mid \|\lambda\|=1\}
$.
\begin{lemma} Let $\lambda_1,\lambda_2\in \SS\cap H_\epsilon$, $\lambda_1\neq \lambda_2$. For $t\in ]0,1[$ put $\lambda_t=(1-t)\lambda_1+t\lambda_2$. Then
\begin{equation}
\label{eq:concave}
q(\lambda_t)=q\left(\frac{\lambda_t}{\|\lambda_t\|}\right)>(1-t)q(\lambda_1)+tq(\lambda_2).
\end{equation}
\end{lemma}
\begin{proof} We have $\|\lambda_t\|<1$. Hence
\[
q(\lambda_t)=\frac{\langle\lambda_t,\varepsilon\rangle}{\|\lambda_t\|}>\langle\lambda_t,\varepsilon\rangle=(1-t)\langle\lambda_1,\varepsilon\rangle
+t\langle\lambda_2,\varepsilon\rangle=(1-t)q(\lambda_1)+tq(\lambda_2).\qedhere
\]
\end{proof}
\begin{corollary} \label{cor:maximum} Assume $\tau_{f}\neq \emptyset$. Then
$q$ attains 
a unique maximum on $\overline{\tau}_{f}\cap \SS$. This maximum is strictly positive.
\end{corollary} 
\begin{proof} $\overline{\tau}_{f}\cap \SS$ is compact and $q$ is continuous so it has at least one maximum on $\overline{\tau}_f\cap \SS$.
Since $q>0$ on $\tau_{f}\neq \emptyset$ this maximum cannot be zero.
If there are two (global) maxima for $\lambda_1\neq \lambda_2\in \overline{\tau}_f\cap \SS$ then by
\eqref{eq:concave} $q((\lambda_1+\lambda_2)/\|\lambda_1+\lambda_2\|)>q(\lambda_1)=q(\lambda_2)$, which is a contradiction.
\end{proof}
As already mentioned above we 
write $\lambda_{f}$ (or $\lambda_F$ if $F$ is the face containing $f$) for the element of $\overline{\tau}_f\cap \SS$ where $q$ attains its maximum. 
By \eqref{eq:barsigma}
\begin{equation*}
\label{eq:note}
f\in \overline{F}_{{}\lambda_{f}}.
\end{equation*}
Since $q$ is
invariant under dilation, $\RR_{>0}\lambda_{f}$ is the half ray in $\overline{\tau}_{f}$ where $q$  takes its maximum
values. We also write 
\[
q_f=
\begin{cases}
q(\lambda_f)&\text{if $\tau_f\neq \emptyset$,}\\
-\infty&\text{otherwise.}
\end{cases}
\]
The following is our main technical result.
\begin{lemma} \label{lem:technical} Assume that $W$ is quasi-symmetric and the generic $T$-stabilizer is finite.
Let $E=X(T)_\RR$, $\Pi=\overline{\Sigma}$.
 Let $f\in\overline{\Sigma}$ be such that $\tau_{f}\neq \emptyset$.
Let $\emptyset\neq \{i_1,\ldots,i_p\}\subset [d]$ be such that $\forall j:\langle \lambda_f,\beta_{i_j}\rangle>0$.
Put $f'=f+2(\beta_{i_1}+\cdots+\beta_{i_p})$. Then
$f'\in \overline{\Sigma}$ and $q_{f'}<q_{f}$.
\end{lemma}
\begin{proof} 
Let $f=\sum_i c_i\beta_i$ with $c_i\in [-1,0]$. Since $f\in \overline{F}_{\lambda_f}$ we have by \eqref{eq:closure1}
\begin{equation}
\label{eq:closure4}
\begin{aligned}
\langle \lambda_f,\beta_i\rangle >0 &\Longrightarrow c_i=-1,\\
\langle \lambda_f,\beta_i\rangle <0 &\Longrightarrow c_i=0.
\end{aligned}
\end{equation}
Or explicitly
\[
f=-\sum_{\langle \lambda_f,\beta_i\rangle>0} \beta_i+\sum_{\langle\lambda_f,\beta_i\rangle=0}
c_i\beta_i.
\]
Then using quasi-symmetry we may write
$
f'=\sum_i c'_i\beta_i
$
with $c'_i\in [-1,0]$, hence $f'\in\overline{\Sigma}$, and moreover 
\begin{equation}
\label{eq:cc}
c'_i=c_i\quad \text{if $\langle \lambda_f,\beta_i\rangle=0$}
\end{equation}
(one may verify this separately for every ray $0\in \ell\subset X(T)_\RR$ and there it is easy).
Assume $q(\lambda_{f'})\ge q(\lambda_f)$
and put $\lambda_t=(1-t)\lambda_f+t\lambda_{f'}$ for $0<t\ll 1$. By \eqref{eq:concave} we have
$q(\lambda_t)>q(\lambda_f)$ as $\lambda_f\neq \lambda_{f'}$. We claim that $\lambda_t\in \overline{\tau}_f$, or equivalently $f\in \overline{F}_{\lambda_t}$,
which is a contradiction with the fact that $q(\lambda_f)$ is the maximum value of $q$ on $\overline{\tau}_f$.
According to \eqref{eq:closure1} we must check
\begin{equation}
\label{eq:closure2}
\begin{aligned}
\langle \lambda_t,\beta_i\rangle>0&\Longrightarrow c_i=-1,\\
\langle \lambda_t,\beta_i\rangle<0&\Longrightarrow c_i=0.
\end{aligned}
\end{equation}
This condition follows from \eqref{eq:closure4} if $\langle \lambda_f,\beta_i\rangle\neq 0$ ($\lambda_t$ is 
close to $\lambda_f$ and hence $\langle \lambda_f,\beta_i\rangle$ and $\langle \lambda_t,\beta_i\rangle$
have the same sign). Therefore we may assume $\langle \lambda_f,\beta_i\rangle=0$. By \eqref{eq:cc}\eqref{eq:closure2} we must
have
\[
\begin{aligned}
\langle \lambda_{f'},\beta_i\rangle>0&\Longrightarrow c'_i=-1,\\
\langle \lambda_{f'},\beta_i\rangle<0&\Longrightarrow c'_i=0,
\end{aligned}
\]
but this follows from \eqref{eq:closure1} and the fact that by definition $f'\in \overline{F}_{\lambda_{f'}}$.
\end{proof}
\begin{remark} \label{rem:easy}
For use below we note that  $f'\in \overline{\Sigma}$ in  Lemma \ref{lem:technical}   would be true  with any $\lambda$ replacing $\lambda_f$ such that
$f\in \overline{F}_\lambda$ and $\forall j:\la \lambda,\beta_{i_j}\ra>0$.
\end{remark}
\subsection{Weyl group action} 
Now we let $\Pi$ be as in \S\ref{sec:someresults} but we assume $E=X(T)_\RR$ and $\Pi$ is $\Wscr$-invariant.
\begin{lemma} \label{lem:weyl}
Let $F$ be a face of $\Pi$ such that $F\cap X(T)^+_\RR\neq \emptyset$. Then $\sigma_F\cap
Y(T)^-_\RR\neq\emptyset$. Moreover
if $\lambda\in \overline{\sigma}_F$
is invariant under the stabilizer of $F$ in $\Wscr$ then $\lambda\in Y(T)^-_{\RR}$.
\end{lemma}
\begin{proof} We prove first $\sigma_F\cap
Y(T)^-_\RR\neq\emptyset$. Let $\lambda\in \sigma_F$; i.e. $F=F_\lambda$.
Let $f\in F\cap  X(T)^+$. Then we have for
  all $\nu\in \Pi$
\begin{equation}
\label{eq:ineg1}
\la \lambda,\nu \ra\ge \langle \lambda ,f\rangle,
\end{equation}
with equality if and only if $\nu\in \overline{F}$. 
Let $w\in \Wscr$ be such that $w\lambda\in Y(T)^-_{\RR}$.
By \cite[Corollary D.3]{SVdB} we find
\begin{equation}
\label{eq:ineg2}
\langle w\lambda ,f\rangle
\le \langle \lambda,f\rangle.
\end{equation}
Combining \eqref{eq:ineg1}\eqref{eq:ineg2} we conclude that for all $\nu\in \Pi$ we have
\begin{equation}
\label{eq:ineg3}
\la w\lambda, w\nu\ra=\la \lambda,\nu \ra\ge \langle w\lambda ,f\rangle.
\end{equation}
If this is an equality then in particular \eqref{eq:ineg1} is an
equality so that $\nu\in \overline{F}$. 

Since we clearly have equality in \eqref{eq:ineg3}
for $\nu=w^{-1} f$ we conclude in particular that
$w^{-1}f\in \overline{F}$ which implies that in our current setting
$w^{-1}F=F$. It follows that $F_{w\lambda}=F_\lambda=F$ and hence in
particular $w\lambda\in \sigma_F\cap Y(T)_\RR^-$.

Let $H\subset \Wscr$ be the stabilizer of $F$. Since in the above proof
$w\in H$, we obtain that
if
$\lambda$ is $H$-invariant then $\lambda\in \sigma_F\cap Y(T)_\RR^-$.
In other words
$
\sigma^H_F\subset Y(T)^-_{\RR}
$.
Taking the closure yields  $\overline{\sigma}^H_F\subset Y(T)^-_{\RR}$ which is
the last statement of the lemma.
\end{proof}
\begin{corollary} \label{cor:lambdaf}
If we are given $\varepsilon\in X(T)^\Wscr_\RR$ as in \S\ref{sec:someresults}  and
$f\in X(T)^+_\RR\cap \Pi$, then $\lambda_f\in Y(T)^-_{\RR}$.
\end{corollary}
\begin{proof} If $F$ is the face of $\Pi$ containing $f$ then 
$\lambda_f=\lambda_F$ is invariant for the stabilizer of $F$ as $\varepsilon\in X(T)^\Wscr_\RR$. It now suffices
to invoke Lemma \ref{lem:weyl}.
\end{proof}
\subsection{Proof of Proposition \ref{prop:mainprop}}
\label{sec:selfcontained}
\begin{enumerate}
\item
We note that the fan $\Sigma_\Pi$ introduced in Proposition \ref{prop:CLS} is invariant
under translation of $\Pi$. Moreover $\lambda_\chi$ and $q_\chi$ are also invariant
under translation. Hence \eqref{it:mainprop1} follows
from Corollary \ref{cor:lambdaf} applied with $f=\chi-\deltaC+\overline{\rho}\in X(T)^+_\RR\cap \Delta_0$.
\item This follows from Lemma \ref{lem:technical} with Remark \ref{rem:easy} and the fact that $\deltaC-\overline{\rho}+\Delta_0$
is invariant under the twisted Weyl group action.
\item 
If $\mu \in X(T)^+_\RR$ then by translation invariance and the fact that $q$ is $\Wscr$-invariant
we obtain $q_{\mu^+}=q_\mu$. Therefore the inequalities $q_\zeta<q_\chi$ follow from Lemma \ref{lem:technical}.
\end{enumerate}

\bibliographystyle{amsalpha}

\end{document}